\documentclass[]{article}
\usepackage{amsmath,amsthm,nicefrac}
\usepackage{bm,amssymb,mathrsfs,url}
\usepackage{setspace}%

\def\al{\alpha}

\def\Ga{\Gamma}
\def\de{\delta}
\def\De{\Delta}
\def\ep{\varepsilon}

\def\ka{\kappa}

\def\ph{\varphi}

\def\wh{\widehat}

\def\pa{\partial}

\renewcommand{\P}{\mathrm{P}}
\newcommand{\E}{\mathrm{E}}
\newcommand{\R}{\mathbf{R}}
\newcommand{\F}{\mathcal{F}}
\renewcommand{\d}{\mathrm{d}}
\newcommand{\e}{\mathrm{e}}
\newcommand{\fbm}{\textnormal{fBm}}
\newcommand{\bfbm}{\textnormal{bi-fBm}}

\newcommand{\lip}{\textnormal{Lip}}

\newtheorem{stat}{Statement}[section]
\newtheorem{proposition}[stat]{Proposition}

\newtheorem{theorem}[stat]{Theorem}
\newtheorem{lemma}[stat]{Lemma}
\theoremstyle{definition}

\newtheorem{remark}[stat]{Remark}

\numberwithin{equation}{section}

\renewcommand{\leq}{\leqslant}
\renewcommand{\ge}{\geqslant}
\renewcommand{\le}{\leqslant}

\begin{document}

\title{Weak existence of a solution to a differential equation driven by  a very rough
        \fbm%
        \thanks{Research supported in part by NSF grant DMS-1307470.}
}
\author{Davar Khoshnevisan\\University of Utah
        \and
        Jason Swanson\\University of Central Florida
        \and
        Yimin Xiao\\Michigan State University
        \and
        Liang Zhang\\Michigan State University
}

\date{
October 15, 2014}
\maketitle

\begin{abstract}
        We prove that if $f:\R\to\R$ is Lipschitz continuous,
        then for every $H\in(0\,,\nicefrac14]$ there exists a probability space
        on which we can construct a fractional Brownian motion $X$ with
        Hurst parameter $H$,
        together with a process $Y$ that: (i) is H\"older-continuous with H\"older
        exponent $\gamma$ for any $\gamma\in(0\,,H)$; and (ii)
        solves the differential equation
        $\d Y_t = f(Y_t)\,\d X_t$. More significantly,
        we describe the law of the stochastic process
        $Y$ in terms of the solution to a non-linear stochastic partial
        differential equation.

\bigskip

        \noindent{\it Keywords:} Stochastic differential equations; rough paths;
        fractional Brownian motion; fractional Laplacian;
        the stochastic heat equation.

\bigskip

        \noindent{\it \noindent AMS 2000 subject classification:}
        60H10; 60G22; 34F05.
\end{abstract}
\section{Introduction}

Let us choose and fix some $T>0$ throughout,
and consider the differential equation
\[
        \d Y_t = f(Y_t) \,\d X_t\qquad(0< t\le T),
        \tag{DE$_0$}
\]
that is driven by a given, possibly-random, signal $X:=\{X_t\}_{t\in[0,T]}$
and is subject to some given initial value $Y_0\in\R$ which we hold
fixed throughout. The sink/source function $f:\R\to\R$ is also fixed throughout, and is
assumed to be Lipschitz continuous, globally, on all of $\R$.

It is well known---and
not difficult to verify from first principles---that 
when the signal $X$ is a Lipschitz-continuous function, then:
\begin{enumerate}
        \item[(i)] The differential equation (DE$_0$) has a solution $Y$ 
                that is itself Lipschitz continuous;
        \item[(ii)] The Radon--Nikod\'ym derivative $\d Y_t/\d X_t$ exists,
                is continuous, and solves $\d Y_t/\d X_t=f(Y_t)$ for  every $0<t\le T$; and 
        \item[(iii)] The solution to (DE$_0$) is unique.
\end{enumerate}
Therefore, the Lebesgue differentiation theorem implies that
we can recast (DE$_0$) equally well as the solution to the following: As
$\varepsilon\downarrow 0$,
\[
        \frac{Y_{t+\varepsilon} - Y_t}{X_{t+\varepsilon} - X_t} = f(Y_t)+o(1),
        \tag{DE}
\]
for almost every $t\in[0\,,T]$.\footnote{To be completely careful,
we might have to define $0\div 0:=0$
in the cases that $X$ has intervals of constancy. But 
with probability one, this will be a moot issue for
the examples that we will be considering soon.} Note that (DE) always has an ``elementary" solution, even when $X$ is assumed only to be continuous. Namely, if $y$ is a solution to the ODE, $y'=f(y)$, and we set $Y_t=y(X_t)$, then $Y_{t+\ep}-Y_t=f(Y_t)(X_{t+\ep}-X_t) + o(|X_{t+\ep}-X_t|)$. Also note that if $Y$ is a solution to (DE) and $\xi$ is a process that is smoother than $X$ in the sense that $\xi_{t+\ep}-\xi_t=o(|X_{t+\ep}-X_t|)$, then $Y+\xi$ is also a solution to (DE).

Differential equations such as (DE$_0$) and/or
(DE) arise naturally also when $X$ is
H\"older continuous with some positive index $\gamma<1$. One
of the best-studied such examples is when $X$ is Brownian motion on the time interval $[0\,,T]$.
In that case, it is very well known that $X$ is H\"older continuous with index
$\gamma$ for any $\gamma<\nicefrac12.$ It is also very well known
that (DE$_0$) and/or
(DE) has infinitely-many strong solutions
\cite{YamadaWatanabe}, and that there is a unique pathwise solution provided 
that we specify what we mean by the stochastic integral $\int_0^tf(Y_s)\,\d X_s$ 
[consider the integrals of It\^o and Stratonovich, for instance]. 

This  view of
stochastic differential equations plays an important role in the pathbreaking work 
\cite{Lyons:94,Lyons:95}  of
T. Lyons who invented his \emph{theory of rough paths}
in order to solve (DE$_0$) when $X$ is rougher than Lipschitz continuous.
Our reduction of (DE) to (DE$_0$) is motivated strongly by Gubinelli's
theory of \emph{controlled rough paths} \cite{Gubinelli}, which we have
learned from a recent paper of Hairer \cite{Hairer}. In the present context, 
Gubinelli's theory of controlled rough paths basically states that if we
could prove \emph{a priori} that the $o(1)$
term in  (DE) has enough structure, then there is a unique solution to (DE),
and hence (DE$_0$).

Lyons' theory builds on older ideas of Fox \cite {Fox} and
Chen \cite{Chen}, respectively
in algebraic differentiation and integration theory, in order
to construct, for a large family of functions $X$,
 ``rough-path integrals'' $\int_0^tf(Y_s)\,\d X_s$
 that are defined uniquely provided that a certain number of ``multiple
stochastic integrals'' of $X$ are pre specified. Armed with a specified
definition of the stochastic integral $\int_0^tf(Y_s)\,\d Y_s$, one can then
try to solve the differential equation (DE) and/or (DE$_0$) pathwise
[that is $\omega$-by-$\omega$]. To date, this program
has been particularly successful when $X$ is H\"older continuous
with index $\gamma\in[\nicefrac13\,,1]$: When $\gamma\in(\nicefrac12\,,1]$
one uses Young's theory of integration; $\gamma=\nicefrac12$ is covered
in essence by martingale theory; and Errami and Russo \cite{ErramiRusso}
and Chapter 5 of 
Lyons and Qian \cite{LyonsQian} both discuss the more difficult case
$\gamma\in[\nicefrac13\,,\nicefrac12)$. There is also mounting evidence that 
one can extend this strategy to cover values 
of $\gamma\in[\nicefrac14\,,1]$---see \cite{%
AlosLeonNualart,AlosMazetNualart,BurdzySwanson,CoutinVictoir,%
CoutinFritzVictoir,GradinaruRussoVallois,RussoVallois}---and
possibly even $\gamma\in(0\,,\nicefrac14)$---see the two recent 
papers by  Unterberger \cite{Unterberger} and
Nualart and Tindel \cite{NualartTindel}.

As far as we know, very little is known about the probabilistic structure of
the solution when $\gamma<\nicefrac12$ [when the solution
is in fact known to exist]. Our goal  is to say something about the
probabilistic structure of {\it a} solution for a concrete, but highly interesting,
family of choices for $X$ in (DE).

A standard fractional Brownian motion [fBm] with Hurst 
parameter $H\in(0,1)$---abbreviated $\fbm(H)$---is a continuous, 
mean-zero Gaussian process 
$X:=\{X_t\}_{t\ge0}$ with $X_0=0$ a.s. and
\begin{equation}\label{fBm:d}
        \E\left( |X_t - X_s|^2\right) = |t-s|^{2H}\qquad(s,t\ge 0).
\end{equation}
Note that fBm$(\nicefrac12)$ is a standard Brownian 
motion. We refer to any constant multiple of a standard fractional 
Brownian motion, somewhat more generally, as fractional Brownian motion
[fBm].

Here, we study the differential equation (DE) in the special case that 
$X$ is $\fbm(H)$ with 
\begin{equation}
        0<H\le \tfrac14.
\end{equation}
It is well known that
\eqref{fBm:d} implies that $X$ is H\"older continuous with index $\gamma$ for
every $\gamma<H$, up to a modification.\footnote{
In other words,  $X\in\cap_{\gamma\in(0,H)}C^\gamma([0\,,T])$ a.s.,
where $C^\gamma([0\,,T])$ denotes as usual the collection of all continuous
functions $f:[0\,,T]\to\R$ such that $|f(t)-f(s)|\le \text{const}\cdot|t-s|^\gamma$ 
uniformly for all $s,t\in[0\,,T]$.}  Since $H\in(0\,,\nicefrac14]$,
we are precisely in the regime where not a great deal is known about (DE).

In analogy with the classical literature on stochastic differential equations
\cite{YamadaWatanabe} the following theorem establishes the ``weak existence''
of a solution to (DE), provided that we interpret the little-$o$ term
in (DE$_0$), somewhat generously, as ``little-$o$ in probability.'' Our theorem
says some things about the law of the solution as well.

\begin{theorem}\label{th:SDE}
        Let $g:\R\to\R$ be Lipschitz continuous uniformly on all of $\R$.
        Choose and fix $H\in(0\,,\nicefrac14]$.
        Then there exists a probability space $(\Omega\,,\F,\P)$ on which
        we can construct a fractional Brownian motion $X$, with Hurst parameter $H$,
        together with a stochastic process $Y\in\cap_{\gamma\in(0,H)} C^\gamma([0\,,T])$
        such that 
        \begin{equation}
                \lim_{\varepsilon\downarrow 0}\sup_{t\in(0,T]}
                \P\left\{ \left| \frac{Y_{t+\varepsilon} - Y_t}{X_{t+\varepsilon} - X_t}
                - g(Y_t)\right|>\delta\right\}=0\qquad\text{for all $\delta>0$}.
        \end{equation}
        Moreover, $Y:=\{Y_t\}_{t\in[0,T]}$ has the same law as $\{\ka_H u_t(0)\}_{t\in[0,T]}$,
        where 
        \begin{equation}\label{kaHdef}
                \ka_H := \left(\frac{(1-2H)\Ga(1-2H)}{2\pi H}\right)^{1/2},
        \end{equation}
        and $u$ denotes the mild solution to the nonlinear
        stochastic partial differential equation,
        \begin{equation}\label{SDEmodified}
                \frac{\pa}{\pa t} u_t(x)  = \frac12(\De_{1/(2-4H)}u_t)(x)
                + \frac{1}{2^{(1-2H)/2}\cdot\kappa_H^2}\, 
                g(\ka_H u_t(x))\dot{W}_t(x),
        \end{equation}
        on $(t\,,x)\in(0\,,T]\times\R$, subject to $u_0(x) \equiv Y_0$ for all $x\in\R$,
        where $\dot{W}$ denotes a space-time white noise, 
        and $\Delta_{\alpha/2} := -(-\Delta)^{\alpha/2}$ 
        denote the fractional Laplace operator, which is usually defined by the property that 
        $(\De_{\al/2}\ph)^{\widehat{\phantom{\i}}}(\xi) = 
        -|\xi|^\al \wh\ph(\xi)$;
        see Jacob \cite[Vol.\ II]{Jacob}. 
\end{theorem}

The preceding can be extended to all of $H\in(0\,,\nicefrac12)$
by replacing, in \eqref{heat} below, the space-time white noise $\dot{W}_t(x)$ by a generalized
Gaussian random field $\psi_t(x)$
whose covariance measure is described by
\begin{equation}
        \text{Cov}(\psi_t(x)\,,\psi_s(y))= \frac{\delta_0(t-s)}{%
        |x-y|^\theta},
\end{equation}
for a suitable choice of  $\theta\in(0\,,1)$. We will not pursue this matter
further here since we do not know how to
address the more immediately-pressing question of uniqueness
in Theorem \ref{th:SDE}. Namely, we do not know a good answer to
the following: ``\emph{What are [necessarily global] non-trivial conditions that
ensure that our solution $Y$ is unique in law}''? 

Throughout this paper, $A_q$ denotes a finite constant that depends critically
only on a [possibly vector-valued] parameter $q$ of interest. We will
not keep track of parameter dependencies for the parameters
that are held fixed throughout; they include 
$\alpha$ and $H$ of \eqref{alpha:H} below, as well as the functions
$g$ [see Theorem \ref{th:SDE}] and $f$ [see \eqref{f} below]. 

The value of $A_q$ might change from line to line, and sometimes even within
the line.

In the absence of interesting parameter dependencies,
we write a generic ``const'' in place of ``$A$.''

We prefer to write $\|\cdots\|_k$ in place of 
$\|\cdot\|_{L^k(\Omega)}$, where $k\in[1\,,\infty)$ can be an arbitary real number. 
That is, for every random variable $Y$, we set
\begin{equation}
        \|Y\|_k := \left\{ \E\left( |Y|^k\right)\right\}^{1/k}.
\end{equation}

On a few occasions we might write $\lip_\varphi$ for the optimal Lipschitz
constant of a function $\varphi:\R\to\R$; that is,
\begin{equation}
        \lip_\varphi := \sup_{-\infty<x<y<\infty}\left| \frac{\varphi(x)-\varphi(y)}{x-y}\right|.
\end{equation}

\section{Some Gaussian random fields}
In this section we recall a decomposition theorem of {Lei} and {Nualart}
\cite{LeiNualart} which will play an important role in this paper; 
see Mueller and Wu \cite{MuellerWu} for a related set of ideas.
We also work out an example that showcases further
the Lei--Nualart theorem.

\subsection{\fbm\ and \bfbm}
Suppose that $H\in(0\,,1)$ and  $K\in(0\,,1]$ are fixed numbers.\footnote{%
Although we are primarily interested in $H\in(0\,,\nicefrac14]$, 
we study the more general case $H\in(0\,,1)$ in this section.} 
A standard bifractional 
Brownian motion, abbreviated as $\bfbm(H\,,K)$, is a continuous mean-zero Gaussian 
process $B^{H,K}:=\{B^{H,K}_t\}_{t\ge 0}$ with $B^{H,K}_0:=0$ a.s.\
and covariance function
\begin{equation}
        \textnormal{Cov}\left(B^{H,K}_t,\,B^{H,K}_{t'}\right) = 
        2^{-K}\left(\left[t^{2H}+(t')^{2H}\right]^K-|t-t'|^{2HK}\right),
\end{equation}
for all $t',t\ge 0$.
Note that $B^{H,1}$ is a fractional Brownian motion with Hurst parameter
$H\in(0\,,1)$. More generally, any constant multiple of 
a standard bifractional Brownian motion will be referred as bifractional Brownian motion.

Bifractional Brownian motion was invented by {Houdr\'e} and {Villa} \cite{HoudreVilla} as 
a concrete example (besides fractional Brownian motion) of a family
of processes that yield natural ``quasi--helices'' in the sense of 
Kahane \cite{Kahane} and/or  ``screw lines'' of classical Hilbert-space theory
\cite{vonNeumannSchoenberg,Schoenberg}. Some sample path properties of 
$\bfbm(H\,,K)$ have been studied by  Russo and Tudor \cite{RT}, Tudor and Xiao \cite{TX}
and Lei and Nualart \cite{LeiNualart}. In particular, the following decomposition
theorem is due to Lei and Nualart  \cite[Proposition 1]{LeiNualart}.

\begin{proposition}
\label{pr:LeiNualart}
        Let $B^{H,K}$ be a $\bfbm(H\,,K)$. There exists a fractional Brownian motion $B^{HK}$
        with Hurst parameter $HK$ and a stochastic 
        process $\xi$ such that $B^{H,K}$ and $\xi$ are independent and, outside a single $\P$-null set,
        \begin{equation}
                B^{H,K}_t=2^{(1-K)/2} B^{HK}_t + \xi_t\qquad\text{for all $t\ge0$}.
        \end{equation}
        Moreover, the process $\xi$ is a centered Gaussian process, with 
        sample functions that are infinitely differentiable 
        on $(0\,, \infty)$ and absolutely continuous on $[0\,,\infty)$.
\end{proposition}

In fact, it is shown in \cite[eq.'s (4) and (5)]{LeiNualart} that we can write
\begin{equation}\label{LeiNualart_explicit}
        \xi_t = \left(\frac{K}{2^K\Ga(1-K)}\right)^{1/2}\int_0^\infty
        \frac{1 - \exp(-st^{2H})}{s^{(1+K)/2}}\,\d W_s,
\end{equation}
where $W$ is a standard Brownian motion that is independent of $B^{H,K}$.

\subsection{The linear heat equation}

Let $\wh{\phantom{m}}$ denote the Fourier transform,
normalized so that for every rapidly-decreasing function $\ph:\R\to\R$,
\begin{equation}
        \wh\ph(\xi) = \int_{-\infty}^\infty \e^{i\xi x}\ph(x)\,\d x
        \qquad(\xi\in\R).
\end{equation}
Let $\Delta_{\alpha/2} := -(-\Delta)^{\alpha/2}$ 
denote the fractional Laplace operator. 

Consider the linear stochastic PDE
\begin{equation}\label{SHE-linear}
        \frac{\pa}{\pa t}v_t (x) = \frac12(\Delta_{\alpha/2} v_t)(x) +  \dot{W}_t(x),
\end{equation}
where $v_0(x)\equiv 0$ and $\dot{W}_t(x)$ denotes space-time white noise;
that is,
\begin{equation}
        \dot{W}_t(x)= \frac{\partial^2W_t(x)}{\partial t\partial x},
\end{equation}
in the sense of generalized random fields \cite[Chapter 2, \S2.4]{GV}, for
a space-time Brownian sheet $W$. 

According to the theory of
Dalang \cite{Dalang:99}, the condition
\begin{equation}\label{cond:Dalang}
        1 < \alpha\le 2
\end{equation}
is necessary and sufficient in order for \eqref{SHE-linear} to have a solution
$v$ that is a random function.
{Lei} and {Nualart} \cite{LeiNualart}
have shown that---in the case that $\alpha=2$---the process $t\mapsto v_t(x)$ is
a suitable \bfbm\ for every fixed $x$. In this section we apply the 
reasoning of \cite{LeiNualart}  to the
present setting in order to show that the same can be said
about the solution to \eqref{SHE-linear} for every possible choice of
$\alpha\in(1\,,2]$.

Let $p_t(x)$ denote the fundamental solution to the fractional heat operator $(\pa/\pa t) - 
\tfrac12\Delta_{\alpha/2}$; that is, the function $(t\,;x\,,y)\mapsto p_t(y-x)$
is the transition probability function for a symmetric stable-$\alpha$ L\'evy process,
normalized as follows (see Jacob \cite[Vol.\ III]{Jacob}):
\begin{equation}\label{eq:chf}
        \hat{p}_t(\xi) = \exp\left(-t|\xi|^\alpha/2\right)\qquad
        \textnormal{($t\ge 0$, $\xi\in\R$).}
\end{equation}
The Plancherel theorem implies the following: For all $t>0$,
\begin{equation}\label{eq:L2:p}
        \|p_t\|_{L^2(\R)}^2 = \frac1{2\pi}\|\wh p_t\|_{L^2(\R)}^2
        = \frac{1}{\pi}\int_0^\infty 
        \e^{-t\xi^\alpha}\,\d\xi
        =\frac{\Gamma(1/\alpha)}{\alpha\pi t^{1/\al}}.
\end{equation}
Let us mention also the following variation: By
the symmetry of the heat kernel, $\|p_t\|_{L^2(\R)}^2
=(p_t*p_t)(0)=p_{2t}(0)$. Therefore, the inversion theorem shows that
\begin{equation}
        p_t(0) = \sup_{x\in\R} p_t(x) =\frac{2^{1/\al}\Ga(1/\al)}{\al\pi t^{1/\al}}
        \qquad(t>0).
\end{equation}

Now we can return to the linear stochastic heat equation
\eqref{SHE-linear}, and write its solution $v$, in mild form, as follows:
\begin{equation}\label{v}
        v_t(x) =  \int_{(0,t)\times\R} p_{t-s}(y-x)\, W(\d s\,\d y).
\end{equation}
It is well known \cite[Chapter 3]{Walsh} 
that $v$ is a continuous, centered Gaussian random field.
Therefore, we combine
\eqref{eq:chf}, and \eqref{eq:L2:p}, using
Parseval's identity, in order to see that
\begin{equation}\begin{split}
        \text{Cov}\left( v_t(x)\,,v_{t'}(x)\right) 
                &= \int_0^{t\wedge t'}\d s\int_{-\infty}^\infty\d y\
                p_{t-s}(y)p_{t'-s}(y)\\
        &= \frac1{2\pi}\int_0^{t\wedge t'}\d s\int_{-\infty}^\infty\d\xi\
                \wh p_{t-s}(\xi)\wh p_{t'-s}(\xi)\\
        &= \frac{\Ga(1/\al)}{\pi\al}\int_0^{t\wedge t'} \left({
                \frac{t + t' - 2s}2}\right)^{-1/\al}\d s.
\end{split}\end{equation}
We use the substitution $r=(t+t'-2s)/2$ and note 
that $(t+t')/2-(t\wedge t')=|t-t'|/2$ in order to conclude that
\begin{equation}
        \text{Cov}\left( v_t(x)\,,v_{t'}(x)\right) 
        = c_\al^22^{(1-\al)/\al}\left(
        |t' + t|^{(\al-1)/\al} - |t' - t|^{(\al-1)/\al} \right),
\end{equation}
where
\begin{equation}\label{A:alpha}
        c_\alpha:=   \left(\frac{\Gamma(1/\alpha)}{\pi(\alpha-1)}
        \right)^{1/2}.
\end{equation}
That is, we have verified the following:

\begin{proposition}\label{pr:bifBM:linear}
        For every fixed $x\in\R$, the stochastic process $t\mapsto c^{-1}_\alpha v_t(x)$
        is a \bfbm$(\nicefrac12\,,(\alpha-1)/\alpha)$, where
        $c_\alpha$ is defined in \eqref{A:alpha}. Therefore,
        Proposition \ref{pr:LeiNualart} allows us to write
        \begin{equation}
                v_t(x) =  c_\al2^{1/(2\al)} X_t + R_t
                \qquad(t\ge 0),
        \end{equation}
        where $\{X_t\}_{t\ge 0}$ is \fbm$((\alpha-1)/(2\alpha))$
        and $\{R_t\}_{t\ge0}$ is a centered Gaussian process that is: 
        \begin{enumerate}
                \item[(i)] Independent of $v_\bullet(x)$;
                \item[(ii)] Absolutely continuous on $[0\,,\infty)$, a.s.; and 
                \item[(iii)] Infinitely differentiable on $(0\,,\infty)$, a.s.
       \end{enumerate}
\end{proposition}

\begin{remark}\label{rem:LeiNualart}
        From now on, we choose $\alpha$ and $H$ according to the following relation:
        \begin{equation}\label{alpha:H}
                \alpha := \frac{1}{1-2H}\
                \textnormal{ equivalently }\ H := \frac{\alpha-1}{2\alpha},
        \end{equation}
        so that Dalang's condition \eqref{cond:Dalang} is
        equivalent to the restriction that $H\in(0\,,\nicefrac14$].
        Propositions \ref{pr:LeiNualart} and \ref{pr:bifBM:linear}
        together show that $t\mapsto v_t(x)$ is a smooth 
        perturbation of a [non-standard] fractional Brownian motion.
        In particular, we may compare \eqref{kaHdef} and \eqref{A:alpha} in
        order to conclude that
        \begin{equation}
                \ka_H = c_\alpha,
        \end{equation}
        thanks to our convention \eqref{alpha:H}.
        \qed
\end{remark}

\begin{remark}\label{R:smooth_moments}
        According to \eqref{LeiNualart_explicit} the process $R_t$ of 
        Proposition \ref{pr:bifBM:linear} can be written as
        \begin{equation}
        R_t = \text{const}\cdot
        \int_0^\infty \frac{1 - \exp\left(-st\right)}{s^{H+(1/2)}}\,\d W_s.
        \end{equation}
        This is a Gaussian process that is $C^\infty$ away from $t=0$,
        and its derivatives are obtained by differentiating under the 
        [Wiener] integral. In particular, the first derivative of $R$,
        away from $t=0$, is
        \begin{equation}
                R_t' = \text{const}\cdot
                \int_0^\infty \frac{\exp\left(-st\right)}{s^{H-(1/2)}}\,\d W_s
                \qquad(t>0).
        \end{equation}
        Consequently, $\{R_q'\}_{q>0}$
        defines a centered Gaussian process, and Wiener's isometry shows that
        $\E(|R_q'|^2)=\text{const}\cdot q^{2H-2}$ for all $q>0$. Therefore,
        \begin{equation}\begin{split}
                \| R_{t+\ep}-R_t\|_k&= A_k\|R_{t+\ep}-R_t\|_2
                        \le A_k\int_t^{t+\ep} \|R'_q\|_2\,\d q\\
                &=A_k\int_t^{t+\ep} q^{H-1}\,\d q
                        \le A_k \, t^{H-1}\ep,
        \end{split}\end{equation}
        uniformly over all $t>0$ and $\ep\in(0\,,1)$.
        \qed
\end{remark}

\section{The non-linear heat equation}
In this section we consider the non-linear stochastic heat equation
\begin{equation}\label{heat}
        \frac{\pa}{\pa t}u_t(x)  = \frac12(\Delta_{\alpha/2}u_t)(x) + 
        f(c_\alpha u_t(x))\dot{W}_t(x)
\end{equation}
on $(t\,,x)\in(0\,,T]\times\R$,
subject to $u_0(x) \equiv Y_0$ for all $x\in\R$, where
$c_\alpha$ was defined in \eqref{A:alpha} and $f:\R\to\R$
is a globally Lipschitz-continuous function.

As is customary \cite[Chapter 3]{Walsh}, we interpret \eqref{heat} as
the non-linear random evolution equation,
\begin{equation}\label{mild}
        u_t(x) = Y_0 + \int_{(0,t)\times\R} p_{t-s}(y-x)
        f(c_\alpha u_s(y))\, W(\d s\,\d y).
\end{equation}
Dalang's condition \eqref{cond:Dalang} implies that the
evolution equation \eqref{mild} has an a.s.-unique random-field solution $u$.
Moreover, \eqref{cond:Dalang} is necessary and sufficient for the
existence of a random-field solution when $f$ is a constant; 
see \cite{Dalang:99}. 
We will need the following technical estimates.

\begin{lemma}\label{lem:moments}
        For all $k\in[2\,,\infty)$ there exists a finite constant $A_{k,T}$ such that:
        \begin{equation}\label{eq:moments}\begin{split}
                &\E\left(|u_t(x)|^k\right) \le A_{k,T};\qquad\textnormal{and}\\
                &\E\left( \left| u_t(x) - u_{t'}(x') \right|^k\right) \le A_{k,T}
                        \left( \vert x-x'\vert^{(\alpha-1)k/2}+\vert t-t'\vert^{(\alpha-1)k/(2\alpha)}
                        \right);
        \end{split}\end{equation}
        uniformly for all $t,t'\in[0\,,T]$ and $x,x'\in\R$.
\end{lemma}
This is well known: The first moment bound can be found explicitly in Dalang \cite{Dalang:99},
and the second can be found in the appendix of Foondun and Khoshnevisan
\cite{FA:whitenoise}. The second can also be shown to follow from the 
moments estimates of  \cite{Dalang:99} and some harmonic analysis. 

Lemma \ref{lem:moments} and the Kolmogorov continuity theorem 
\cite[Theorem 4.3, p.\ 10]{Minicourse}
together imply that $u$ is continuous up to a modification.
Moreover, \eqref{alpha:H}
and Kolmogorov's continuity theorem imply that for every $x\in\R$,
\begin{equation}\label{eq:u:in:C}
        u_\bullet(x)\in\bigcap_{\gamma\in(0,H)} C^\gamma([0\,,T]).
\end{equation}

\section{An approximation theorem}

The following is the main technical contribution of this paper. 
It bears the same spirit as the argument made by Hairer et al. for the rough Burgers-like equations \cite{HairerWeber}.
Recall that $v$ denotes the solution to the linear stochastic heat equation
\eqref{SHE-linear}, and has the integral representation \eqref{v}.

\begin{theorem}\label{th:main}
        For every $k\in[2\,,\infty)$
        there exists a finite constant $A_{k,T}$ such that 
        uniformly for all $\varepsilon\in(0\,,1)$, $x\in\R$, and $t\in[0\,,T]$,
        \begin{equation} \label{Eq:main1}
                \E\left( \left| u_{t+\varepsilon}(x) - u_t(x) - f(c_\alpha u_t(x))\cdot
                \left\{ v_{t+\varepsilon}(x) -v_t(x) \right\}\right|^k\right)
                \le A_{k,T}\,\varepsilon^{\mathcal{G}_H k},
        \end{equation}
        where 
        \begin{equation}\label{G}
                \mathcal{G}_H:= \frac{2H}{1+H}.
        \end{equation}
\end{theorem}

\begin{remark}\label{rem:GH>H}
        Since $0<H\le\nicefrac14$, it follows that 
        \begin{equation}\label{eq:G/H}
                \frac{8}{5} \le \frac{\mathcal{G}_H}{H} <2.
        \end{equation}
        We do not know whether the fraction $\nicefrac{8}{5}=1.6$
        is a meaningful quantity or a byproduct of the particulars of our method.
        For us the relevant matter is that \eqref{eq:G/H} is a good enough estimate to ensure
        that $\mathcal{G}_H/H>1$; the strict inequality will play an important role
        in the sequel.\qed
\end{remark}

Theorem \ref{th:main} is in essence an analysis of the temporal
increments of $u_\bullet(x)$. Thanks to \eqref{mild},
we can write those increments as
\begin{equation}\label{eq:u-u}
        u_{t+\varepsilon}(x) - u_t(x) := \mathscr{J}_1 + \mathscr{J}_2,
\end{equation}
where
\begin{equation}\begin{split}
        \mathscr{J}_1 & := \int_{(0,t)\times\R} \left[ p_{t+\varepsilon-s}(y-x)-p_{t-s}(y-x)\right]
                f(c_\alpha u_s(y))\, W(\d s\,\d y);\\
        \mathscr{J}_2 &:= \int_{(t,t+\varepsilon)\times\R} p_{t+\varepsilon-s}(y-x)
                f(c_\alpha u_s(y))\, W(\d s\,\d y).
                \label{eq:J1J2}
\end{split}\end{equation}

Our proof of Theorem \ref{th:main} proceeds by analyzing $\mathscr{J}_1$
and $\mathscr{J}_2$ separately. Let us begin with the latter quantity,
as it is easier to estimate than the former term.

\subsection{Estimation of $\mathscr{J}_2$}
Define
\begin{equation}\label{eq:J2tilde}
        \widetilde{\mathscr{J}}_2 := f(c_\alpha u_t(x))\cdot\int_{(t,t+\varepsilon)\times\R} 
        p_{t+\varepsilon-s}(y-x)\, W(\d s\,\d y).
\end{equation}

\begin{proposition}\label{pr:J2}
        For every $k\in[2\,,\infty)$ there exists a finite constant
        $A_{k,T}$ such that for all $\varepsilon\in(0\,,1)$,
        \begin{equation}
                \sup_{x\in\R}\sup_{t\in[0,T]}
                \E\left(\left| \mathscr{J}_2 - \widetilde{\mathscr{J}}_2\right|^k\right)
                \le A_{k,T}\,\varepsilon^{2Hk}.
        \end{equation}
\end{proposition}

We split the proof in 2 parts: First we show
that $\mathscr{J}_2\approx\mathscr{J}_2'$ in $L^k(\Omega)$, where
\begin{equation}
        \mathscr{J}_2' := \int_{(t,t+\varepsilon)\times\R} p_{t+\varepsilon-s}(y-x)
        f(c_\alpha u_s(x))\,W(\d s\,\d y).
\end{equation}
After that we will verify that $\mathscr{J}_2'\approx\widetilde{\mathscr{J}}_2$ in
$L^k(\Omega)$. 
Proposition \ref{pr:J2} follows immediately from Lemmas \ref{lem:J2-J2'}
and \ref{lem:J2'-J2tilde} below and Minkowski's inequality. Therefore,
we will state and prove only those two lemmas.

\begin{lemma}\label{lem:J2-J2'}
        For all $k\in[2\,,\infty)$ there exists a finite constant $A_{k,T}$ such that
        uniformly for all $\varepsilon\in(0\,,1)$,
        \begin{equation}
                \sup_{x\in\R}\sup_{t\in[0,T]}
                \E\left( \left| \mathscr{J}_2 - \mathscr{J}_2'\right|^k\right) \le A_{k,T}
                \varepsilon^{2Hk}.
        \end{equation}
\end{lemma}

\begin{proof}
        The proof will use a 
        particular form of the Burkholder--Davis--Gundy (BDG)
        inequality \cite[Lemma 2.3]{ConusKh}. Since we will make
        repeated use of this inequality throughout, let us recall it first.
        
        For every $t\ge 0$, let  $\mathscr{F}_t^0$
        denote the sigma-algebra generated by 
        every Wiener integral of the form
        $\int_{(0,t)\times\R}\varphi_s(y)\, W(\d s\,\d y)$
        as $\varphi$ ranges over all elements of $L^2(\R_+\times\R)$. 
        We complete every such sigma-algebra, and make the
        filtration $\{\mathscr{F}_t\}_{t\ge0}$ right continuous
        in order to obtain the ``Brownian filtration'' $\mathscr{F}$
        that corresponds to the white noise $\dot{W}$.
        
        Let $\Phi:=\{\Phi_t(x)\}_{t\ge 0,x\in\R}$
        be a predictable random field with respect to $\mathscr{F}$.
        Then, for every real number $k\in[2\,,\infty)$, we have the
        following BDG inequality:
        \begin{equation}\label{BDG}
                \left\| \int_{(0,t)\times\R} \Phi_s(y)\, W(\d s\,\d y)\right\|_k^2
                \le 4k\int_0^t \d s\int_{-\infty}^\infty\d y\
                \|\Phi_s(y)\|_k^2.
        \end{equation}
        
        The BDG inequality \eqref{BDG}
        and eq.\ \eqref{mild} together imply that
        \begin{align}\notag
                &\left\| \mathscr{J}_2-\mathscr{J}_2'\right\|_{L^k(\Omega)}^2\\\notag
                &\qquad\le  4k\int_t^{t+\varepsilon}
                        \d s\int_{-\infty}^\infty\d y\ \left[ p_{t+\varepsilon-s}(y-x)\right]^2
                        \left\| f(c_\alpha u_s(y)) -f(c_\alpha u_s(x))\right\|_k^2\\
                &\qquad\le 4kc_\alpha^2\lip^2_f\cdot\int_t^{t+\varepsilon}\d s
                        \int_{-\infty}^\infty\d y\
                        \left[ p_{t+\varepsilon-s}(y-x)\right]^2\left\| u_s(y) - u_s(x)\right\|_k^2
                        \label{eq:JJ2}\\\notag
                &\qquad\le A_{k,T}\int_0^\varepsilon\d s\int_{-\infty}^\infty\d y\
                        \left[ p_s(y)\right]^2\left( |y|^{\alpha-1}\wedge 1\right).
        \end{align}
        The last inequality uses both moment inequalities of Lemmas \ref{lem:moments}.
        Furthermore, measurability issues do not arise, since the solution to
        \eqref{mild} is continuous in the time variable $t$ and adapted to the
        Brownian filtration $\mathscr{F}$.
        
        In order to proceed from here,
        we need to recall two basic facts about the transition functions of stable processes:
        First of all, 
        \begin{equation}\label{scaling}
                p_s(y) = s^{-1/\alpha} p_1\left( |y|/s^{1/\alpha} \right)
                \qquad\textnormal{for all $s>0$ and $y\in\R$}.
        \end{equation}
        This fact is a consequence of scaling and symmetry; see \eqref{eq:chf}. 
        We also need to know the  fact that
        $p_1(z) \le \textnormal{const}\cdot (1+|z|)^{-(1+\alpha)}$ for all $z\in\R$
        \cite[Proposition 3.3.1, p.\ 380]{Kh}, whence
        \begin{equation}
                p_s(y) \le \textnormal{const}\times\begin{cases}
                        s^{-1/\alpha}&\textnormal{if $|y|\le s^{1/\alpha}$},\\
                        s|y|^{-(1+\alpha)}&\textnormal{if $|y|>s^{1/\alpha}$}.
                \end{cases}
        \end{equation}
        Consequently,
        \begin{equation}\begin{split}
                &\int_0^\varepsilon\d s\int_0^1\d y\ \left[ p_s(y)\right]^2
                        \left( y^{\alpha-1}\wedge1\right)\\
                &\hskip.27in\le \text{const}\cdot\left({\int_0^\varepsilon s^{-2/\alpha}\,\d s\int_0^{s^{1/\alpha}}
                        y^{\al-1}\,\d y+
                        \int_0^\varepsilon s^2\,\d s\int_{s^{1/\alpha}}^1 y^{-3-\alpha}\,\d y}\right)\\
                &\hskip.27in\le \text{const}\cdot \varepsilon^{2(\alpha-1)/\alpha}.
        \end{split}\end{equation}
        We obtain the following estimate by similar means:
        \begin{align}
                \int_0^\varepsilon\d s\int_1^\infty\d y\ \left[ p_s(y)\right]^2
                        \left( y^{\alpha-1}\wedge1\right)
                &\le\textnormal{const}\cdot\int_0^\varepsilon
                        s^2\, \d s\int_1^\infty y^{-2-2\alpha}\,\d y\\\notag
                &= \textnormal{const}\cdot \varepsilon^3\\\notag
                &\le\text{const}\cdot\varepsilon^{2(\al-1)/\alpha},
        \end{align}
        uniformly for all $\varepsilon\in(0\,,1)$.
        Since $p_s(y)=p_s(-y)$ for all $s>0$ and $y\in\R$,
        the preceding two displays and  \eqref{eq:JJ2}
        together imply that
        \begin{equation}
                \| \mathscr{J}_2-\mathscr{J}_2'\|_{L^k(\Omega)}^2\le
                \text{const}\cdot \varepsilon^{2(\al-1)/\al}.
        \end{equation}
        We may conclude the lemma from this inequality,
        using our convention about $\alpha$ and 
        $H$; see \eqref{alpha:H}.
\end{proof}

In light of Lemma \ref{lem:J2-J2'}, Proposition \ref{pr:J2} follows
at once from

\begin{lemma}\label{lem:J2'-J2tilde}
        For all $k\in[2\,,\infty)$ there exists a finite constant $A_{k,T}$ such that
        uniformly for all $\varepsilon\in(0\,,1)$,
        \begin{equation}
                \sup_{x\in\R}\sup_{t\in[0,T]}
                \E\left( \left| \mathscr{J}_2' - \widetilde{\mathscr{J}}_2\right|^k\right) \le A_{k,T}
                \,\varepsilon^{2Hk}.
        \end{equation}
\end{lemma}

\begin{proof}
        We apply the BDG inequality \eqref{BDG},
        as we did in the derivation of \eqref{eq:JJ2}, in order
        to see that
        \begin{align}\notag
                &\left\| \mathscr{J}_2' - \widetilde{\mathscr{J}}_2\right\|_{L^k(\Omega)}^2 \notag \\
                &\qquad \le 4kc_{\alpha}^2\lip_f^2\int_t^{t+\varepsilon}\d s\int_{-\infty}^\infty\d y\
                        \left[ p_{t+\varepsilon-s}(y)\right]^2 \left\| u_s(x)
                        - u_t(x)\right\|_{L^k(\Omega)}^2\\
                &\qquad \le A_{k,T}\int_t^{t+\varepsilon}
                        \left\| p_{t+\varepsilon-s}\right\|_{L^2(\R)}^2 \left| s-t\right|^{(\alpha-1)/\alpha}
                        \,\d s. \notag
        \end{align}
        Therefore, \eqref{eq:L2:p} and a change of variables together show us that
        the preceding quantity
        is bounded above by 
        \begin{equation}
                A_{k,T}\int_0^\varepsilon s^{(\alpha-1)/\alpha}
                (\varepsilon-s)^{-1/\alpha}\,\d s
                =A_{k,T} \varepsilon^{2(\alpha-1)/\alpha}.
        \end{equation}
        The lemma follows from this and our convention
        \eqref{alpha:H} about the relation between $\al$ and $H$.
\end{proof}

\subsection{Estimation of $\mathscr{J}_1$ and proof of Theorem \ref{th:main}}

Now we turn our attention to the more interesting term $\mathscr{J}_1$ in
the decomposition \eqref{eq:J1J2}. The following localization argument paves 
the way for a successful analysis of $\mathscr{J}_1$: 
$p_t(x)\,\d x\approx \delta_0(\d x)$
when $t\approx 0$; therefore
one might imagine that there is a small regime of values of $s\in(0\,,t)$
such that $p_{t+\varepsilon-s}(y-x)-p_{t-s}(y-x)$ is highly localized [big within the regime,
and significantly smaller outside that regime].
Thus, we choose and fix a parameter $a\in(0\,,1)$---whose optimal value 
will be made explicit later on in \eqref{a}---and write
\begin{equation}\label{J1=J1'+J1tilde}
        \mathscr{J}_1 = \mathscr{J}_{1,a} + \mathscr{J}_{1,a}',
\end{equation}
where
\begin{align}\notag
        \mathscr{J}_{1,a} &:= \int_{(0,t-\varepsilon^a)\times\R}
                \left[ p_{t+\varepsilon-s}(y-x)-p_{t-s}(y-x)\right]
                f(c_\alpha u_s(y))\, W(\d s\,\d y),\\ \\ \notag
        \mathscr{J}_{1,a}' &:= \int_{(t-\varepsilon^a,t)\times\R}
                \left[ p_{t+\varepsilon-s}(y-x)-p_{t-s}(y-x)\right]
                f(c_\alpha u_s(y))\, W(\d s\,\d y).
\end{align}

We will prove that  the quantity $\mathscr{J}_{1,a}$ is  small 
as long as we choose $a\in(0\,,1)$ carefully; 
that is, $\mathscr{J}_1\approx \mathscr{J}_{1,a}'$ for a
good choice of $a$. And because $s\in(t-\varepsilon^a,t)$ is approximately
$t$, then we might expect that
$f(u_s(y)))\approx f(u_t(y))$ [for that correctly-chosen $a$], and
hence $\mathscr{J}_1\approx\mathscr{J}_{1,a}''$, where
\begin{equation}\label{eq:J1a''}\hskip-.2in
        \mathscr{J}_{1,a}'' := \int_{(t-\varepsilon^a,t)\times\R}
        \left[ p_{t+\varepsilon-s}(y-x)-p_{t-s}(y-x)\right]
        f(c_\alpha u_t(y))\, W(\d s\,\d y).
\end{equation}
Finally, we might notice that $p_{t+\varepsilon-s}$ and $p_{t-s}$
both act as point masses when $s\in(t-\varepsilon^a,t)$, and therefore
we might imagine that $\mathscr{J}_1\approx \mathscr{J}_{1,a}''
\approx\widetilde{\mathscr{J}}_{1,a}$,
where
\begin{equation}\label{eq:J1atilde}\hskip-.2in
        \widetilde{\mathscr{J}}_{1,a} :=  f(c_\alpha u_t(x))\cdot
        \int_{(t-\varepsilon^a,t)\times\R}
        \left[ p_{t+\varepsilon-s}(y-x)-p_{t-s}(y-x)\right]W(\d s\,\d y).
\end{equation}
All of this turns out to be true; it remains to find the correct choice[s] for the parameter $a$
so that the errors in the mentioned approximations remain sufficiently small
for our later needs.  Recall the parameter $\mathcal{G}_H$ from
\eqref{G}. Before we continue, let us first document the end result of this forthcoming
effort. We will prove it subsequently. 

\begin{proposition}\label{pr:J1}
        For every $T>0$ and $k\in[2\,,\infty)$
        there exists a finite constant $A_{k,T}$ such that 
        uniformly for all $\varepsilon\in(0\,,1)$, $x\in\R$, and $t\in[0\,,T]$,
        \begin{align}\notag
                &\E\left(\left| \mathscr{J}_1 -f(c_\alpha u_t(x))\cdot
                        \int_{(0,t)\times\R}
                        \left[ p_{t+\varepsilon-s}(y-x)-p_{t-s}(y-x)\right]W(\d s\,\d y)\right|^k\right)\\
                &\hskip3in\le A_{k,T}\,\varepsilon^{\mathcal{G}_H k}.
        \end{align}
\end{proposition}

Thanks to \eqref{eq:u-u} and Minkowski's inequality,
Theorem \ref{th:main} follows easily from Propositions
\ref{pr:J2} and \ref{pr:J1}. It remains to prove Proposition \ref{pr:J1}.

We begin with a sequence of lemmas that
make precise the various formal appeals  to ``$\approx$'' in the preceding discussion.
As a first step in this direction, let us dispense with the ``small'' term $\mathscr{J}_{1,a}$.

\begin{lemma}\label{lem:J1a}
        For all $k\in[2\,,\infty)$ and $a\in(0\,,1)$ there exists a finite constant $A_{a,k,T}$ such that
        uniformly for all $\varepsilon\in(0\,,1)$,
        \begin{equation}
                \sup_{x\in\R}\sup_{t\in[0,T]}\E\left(\left| \mathscr{J}_{1,a}\right|^k\right)
                \le A_{a,k,T}\, \varepsilon^{[1-a(1-H)]k}.
        \end{equation}
\end{lemma}

\begin{proof}
        We can modify the argument that led to \eqref{eq:JJ2},
        using  the BDG inequality \eqref{BDG}, in order to yield
        \begin{align}\notag
                &\left\| \mathscr{J}_{1,a}\right\|_{L^k(\Omega)}^2\\\notag
                &\le 4k \int_0^{t-\varepsilon^a}\d s\int_{-\infty}^\infty\d y\
                        \left[ p_{t+\varepsilon-s}(y-x)-p_{t-s}(y-x)\right]^2
                        \left\| f(c_\alpha u_s(y))\right\|_{L^k(\Omega)}^2\\
                &\le A_{k,T}\int_{\varepsilon^a}^T\d s\int_{-\infty}^\infty\d y\
                        \left[ p_{s+\varepsilon}(y)-p_s(y)\right]^2.
        \end{align}
        We  first bound $\int_{\varepsilon^a}^T\d s$ from above by
        $\e^T\cdot\int_{\varepsilon^a}^\infty \e^{-s}\,\d s$, and then apply \eqref{eq:chf} and
        Plancherel's formula in order to deduce the following bounds:
        \begin{equation}\begin{split}
                \left\| \mathscr{J}_{1,a}\right\|_{L^k(\Omega)}^2
                        &\le A_{k,T}\int_{\varepsilon^a}^\infty\e^{-s}\,\d s\int_{-\infty}^\infty\d \xi\
                        \e^{-2s|\xi|^\alpha}\left| 1 - \e^{-\varepsilon|\xi|^\alpha}\right|^2\\
                &\le A_{k,T}\int_{\varepsilon^a}^\infty\e^{-s}\,\d s\int_0^\infty\d \xi\
                        \e^{-2s\xi^\alpha}\left( 1\wedge\varepsilon^2\xi^{2\alpha}\right)\\
                &= A_{k,T} \int_0^\infty\left( 1\wedge\varepsilon^2\xi^{2\alpha}\right)
                        \e^{-2\varepsilon^a\xi^\alpha}
                        \frac{\d\xi}{1+\xi^\alpha},
        \end{split}\end{equation}
        since $0\le 1-\e^{-z}\le 1\wedge z$ for all $z\ge 0$. 
        Clearly,
        \begin{align}\notag
                \int_0^{\varepsilon^{-1/\al}}
                        \left( 1\wedge\varepsilon^2\xi^{2\al}\right)
                        \e^{-2\varepsilon^a\xi^\al} \frac{\d\xi}{1+\xi^\alpha}
                        &\le \varepsilon^2\int_0^{\varepsilon^{-1/\al}}\xi^\al
                        \e^{-2\varepsilon^a\xi^\al}\d\xi\\\notag
                &=      \varepsilon^{(\al-1)/\al}\int_0^1 x^\al \exp\left(
                        -\frac{2x^\alpha}{\varepsilon^{1-a}}\right)\,\d x\\\notag
                &\le\varepsilon^{(\al-1)/\al}\int_0^\infty x^\al \exp\left(
                        -\frac{2x^\alpha}{\varepsilon^{1-a}}\right)\,\d x\\
                &=\text{const}\cdot\varepsilon^{(2\al-a-\al a)/\alpha}.
        \end{align}
        Furthermore, 
        \begin{align}
                \int_{\varepsilon^{-1/\al}}^\infty
                        \left( 1\wedge\varepsilon^2\xi^{2\al}\right)
                        \e^{-2\varepsilon^a\xi^\al} \frac{\d\xi}{1+\xi^\alpha}
                        &\le \int_{\varepsilon^{-1/\al}}^\infty
                        \e^{-2\varepsilon^a\xi^\al}\d\xi\\\notag
                &\le \text{const}\cdot \exp\left( -2 \varepsilon^{-(1-a)}\right),
        \end{align}
        uniformly for all $\varepsilon\in(0\,,1)$. The  preceding two paragraphs
        together imply that
        \begin{equation}
                \E\left(\left| \mathscr{J}_{1,a}\right|^k\right)
                \le A_{a,k,T}\, \varepsilon^{(2\al - a - a\al)k/(2\al)},
        \end{equation}
        which proves the lemma, due to the relation \eqref{alpha:H} between $H$ and $\al$.
\end{proof}

\begin{lemma}\label{lem:J1a'-J1a''}
        For all $k\in[2\,,\infty)$ and $a\in(0\,,1)$
        there exists a finite constant $A_{a,k,T}$ such that
        uniformly for all $\varepsilon\in(0\,,1)$,
        \begin{equation}
                \sup_{x\in\R}\sup_{t\in[0,T]}
                \E\left(\left| \mathscr{J}_{1,a}'-\mathscr{J}_{1,a}''\right|^k\right)
                \le A_{a,k,T}\,\varepsilon^{2aHk}.
        \end{equation}
\end{lemma}

\begin{proof}
        We proceed as we did for \eqref{eq:JJ2}, using the BDG inequality
        \eqref{BDG}, in order to find that
        \begin{align}\label{eq:J'-J''}
                \left\| \mathscr{J}_{1,a}'-\mathscr{J}_{1,a}''\right\|_{L^k(\Omega)}^2
                        &\le A_{k,T}\cdot \int_0^{\varepsilon^a} s^{(\alpha-1)/\alpha}
                        \|p_{s+\varepsilon} - p_s  \|_{L^2(\R)}^2\,\d s\\\notag
                &= A_{k,T}\varepsilon^{(2\alpha-1)/\alpha}
                        \cdot\int_0^{\varepsilon^{a-1}} r^{(\al-1)/\al}
                        \| p_{\varepsilon(1+r)}-p_{\varepsilon r}\|_{L^2(\R)}^2\,\d r,
        \end{align}
        after a change of variables $[r:=s/\varepsilon]$.
        The scaling property \eqref{scaling} can be written in the following form:
        \begin{equation}\label{scaling1}
                p_{\varepsilon\tau}(y) = \varepsilon^{-1/\alpha}p_\tau(y/\varepsilon^{1/\alpha}),
        \end{equation}
        valid for all $\tau,\varepsilon>0$ and $y\in\R$. Consequently,
        \begin{equation}
                \| p_{\varepsilon(1+r)}-p_{\varepsilon r}
                \|_{L^2(\R)}^2 =
                \varepsilon^{-1/\alpha}\cdot\|p_{1+r}-p_r\|_{L^2(\R)}^2.
        \end{equation}
        
        Eq.\ \eqref{eq:chf} and the Plancherel theorem together imply that
        \begin{equation}\begin{split}
                \|p_{1+r} - p_r  \|_{L^2(\R)}^2 &= \frac{1}{2\pi}\int_{-\infty}^\infty
                        \e^{-r|z|^\alpha}\left( 1 - \e^{-|z|^\alpha/2}\right)^2\d z\\
                &\le\int_0^\infty \e^{-rz^\alpha}\,\d z\\
                &= \frac{\Ga(1/\al)}{\al r^{1/\al}},
        \end{split}\end{equation}
        for all $r>0$. Therefore, \eqref{eq:J'-J''} implies that
        \begin{equation}
                \left\| \mathscr{J}_{1,a}'-\mathscr{J}_{1,a}''\right\|_{L^k(\Omega)}^2
                \le A_{k,T}\varepsilon^{2(\alpha-1)/\alpha}\cdot \int_0^{\varepsilon^{a-1}}
                r^{(\alpha-2)/\alpha}\,\d r,
        \end{equation}
        which readily implies the lemma. 
\end{proof}

\begin{lemma}\label{lem:J1a''-J1atilde}
        For all $k\in[2\,,\infty)$ and $a\in(0\,,1)$ there exists a finite constant $A_{a,k,T}$ such that
        uniformly for all $\varepsilon\in(0\,,1)$,
        \begin{equation}
                \sup_{x\in\R}\sup_{t\in[0,T]}
                \E\left(\left| \mathscr{J}_{1,a}''-\widetilde{\mathscr{J}}_{1,a}\right|^k\right)
                \le A_{k,T}\,\varepsilon^{2aHk}.
        \end{equation}
\end{lemma}

\begin{proof}
        We proceed as we did for \eqref{eq:JJ2}, apply the
        BDG inequality \eqref{BDG}, and  obtain the following bounds:
        \begin{align}\notag             
                &\left\| \mathscr{J}_{1,a}''-\widetilde{\mathscr{J}}_{1,a}\right\|_{L^k(\Omega)}^2\\
                        \notag  
                &\le A_k \int_{t-\varepsilon^a}^t\d s\int_{-\infty}^\infty\d y\
                        \left[ p_{t+\varepsilon-s}(y-x) - p_{t-s} (y-x) \right]^2\left\|
                        u_t(y) - u_t(x)\right\|_{L^k(\Omega)}^2\\
                &\le A_{k,T}\int_0^{\varepsilon^a}\,\d s\int_{-\infty}^\infty\d y\
                        \left[ p_{s+\varepsilon}(y) - p_s (y) \right]^2\left(|y|^{\alpha-1}\wedge
                        1\right)\\ \notag
                &\le A_{k,T} \varepsilon\int_0^{\varepsilon^{a-1}}\d r\int_{-\infty}^\infty\d y\
                        \left[ p_{\varepsilon(r+1)}(y)-p_{\varepsilon r}(y)\right]^2 |y|^{\alpha-1}.
        \end{align}
        Thanks to the scaling property \eqref{scaling1},
        we may obtain the following after a change of variables
        $[w:=y/\varepsilon^{1/\alpha}]$:
        \begin{align}
                &\left\| \mathscr{J}_{1,a}''-\widetilde{\mathscr{J}}_{1,a}\right\|_{L^k(\Omega)}^2\\
                        \notag  
                &\hskip.8in\le A_{k,T} \varepsilon^{2(\al-1)/\al}
                        \int_0^{\varepsilon^{a-1}}\d r\int_{-\infty}^\infty\d w
                        \left[ p_{r+1}(w)-p_r(w)\right]^2 |w|^{\alpha-1}.
        \end{align}
        Next we notice that
        \begin{align}\notag 
                \int_0^{\varepsilon^{a-1}}\d r\int_0^\infty\d w \left[ p_r(w)\right]^2 w^{\alpha-1} 
                        &=\int_0^{\varepsilon^{a-1}} r^{-2/\alpha}\,\d r\int_0^\infty\d w\left[
                        p_1(w/r^{1/\alpha})\right]^2 w^{\alpha-1}  \\\notag
                & =\int_0^{\varepsilon^{a-1}} r^{(\alpha-2)/\alpha}\,\d r\int_0^\infty\d x\left[
                        p_1(x)\right]^2 x^{\alpha-1}, \\
                & \leq \text{const} \cdot \varepsilon^{2(a-1)(\alpha-1)/\alpha}, 
        \end{align}
        where the last inequality uses the facts that:
        (i) $\alpha>1$; and (ii) $p_1(x)\le
        \text{const}\cdot (1+|x|)^{-1-\alpha}$ (see \cite[Proposition 3.3.1, p.\ 380]{Kh}).
        Therefore,
        \begin{equation}
          \left\| \mathscr{J}_{1,a}''-\widetilde{\mathscr{J}}_{1,a}\right\|_{L^k(\Omega)}^2 \leq A_{k,T} \varepsilon^{2a(\alpha-1)/\alpha},
        \end{equation}
        which proves the lemma, due to the relation \eqref{alpha:H} between $H$ and $\al$.
\end{proof}

\begin{proof}[Proof of Proposition \ref{pr:J1}]
        So far, the parameter $a$ has been an arbitrary real number in $(0\,,1)$.
        Now we choose and fix it as follows:
        \begin{equation}\label{a}
                a := \frac{1}{1+H}.
        \end{equation}
        Thus, for this particular choice of $a$,
        \begin{equation}
                1-a(1-H)=2aH  = \mathcal{G}_H,
        \end{equation}
        where $\mathcal{G}:=2H/(1+H)$ was defined in \eqref{G}.
        Because $\mathcal{G}_H<2H$ and because of \eqref{J1=J1'+J1tilde}, 
        Lemmas \ref{lem:J1a}, \ref{lem:J1a'-J1a''}, 
        and  \ref{lem:J1a''-J1atilde} together imply that, for this choice of $a$,
        \begin{equation}
                \E\left(\left| \mathscr{J}_1 -  \widetilde{\mathscr{J}}_{1,a}
                \right|^k\right) \le  A_{k,T}\,\varepsilon^{\mathcal{G}_H k},
        \end{equation}
        uniformly for all $\varepsilon\in(0\,,1)$, $x\in\R$, and $t\in[0\,,T]$. 
        Thanks to the definition \eqref{eq:J1atilde}  of
        $\widetilde{\mathscr{J}}_{1,a}$, it suffices to demonstrate the following with the
        same parameter dependencies as above:
        \begin{equation}
                \E\left(\left|  \widetilde{\mathscr{J}}_{1,a} - f(c_\alpha u_t(x))\cdot
                \Lambda([0\,,t])\right|^k\right) \le A_{k,T}\,\varepsilon^{\mathcal{G}_H k};
        \label{goal1}\end{equation}
        where $\Lambda(Q) := \int_{Q\times\R}[ 
        p_{t+\varepsilon-s}(y-x)-p_{t-s}(y-x)]W(\d s\,\d y)$
        for every interval $Q\subset[0\,,T]$. 
        
        Because $\widetilde{\mathscr{J}}_{1,a}
        =f(c_\alpha u_t(x))\times\Lambda([t-\varepsilon^a\,,t])$, Lemma
        \ref{lem:moments} shows that the left-hand side of
        \eqref{goal1} 
        \begin{equation}
                \E\left(\left|  \widetilde{\mathscr{J}}_{1,a} - f(c_\alpha u_t(x))\cdot
                \Lambda([0\,,t])\right|^k\right) \le A_{k,T}
                \sqrt{\E\left(\left| \Lambda([0\,,t-\varepsilon^a])
                \right|^{2k}\right)}.
        \end{equation}
        Since $\Lambda([0\,,t-\varepsilon^a])$
        is the same as the quantity $\mathscr{J}_{1,a}$ in the case that $f\equiv 1$,
        we may apply Lemma \ref{lem:J1a} to the linear equation \eqref{SHE-linear}
        with $f\equiv 1$ in order to see that 
        \begin{equation}
                \sqrt{\E\left(\left|
                \Lambda([0\,,t-\varepsilon^a])\right|^{2k}\right)}
                \le A_{k,T}\, \varepsilon^{^{\mathcal{G}_H k}},
        \end{equation}
        which implies \eqref{goal1}.
\end{proof}

\section{Proof of Theorem \ref{th:SDE}}
We conclude this article by proving Theorem \ref{th:SDE}. 

Let us define a Lipschitz-continuous function $f$ by 
\begin{equation}\label{f}
        f(x) :=\frac{2^H}{\kappa_H^2\sqrt 2}\, g(x)\qquad(x\in\R),
\end{equation}
where $\kappa_H$ was defined in \eqref{kaHdef}. Let us
also define a stochastic process
\begin{equation}
        Y_t := c_\alpha u_t(0)\qquad(t\ge 0),
\end{equation}
where the constant $c_\alpha[=\kappa_H]$ was defined in \eqref{A:alpha} and $u$ 
denotes the solution to the stochastic PDE \eqref{heat}.
Because of Remark \ref{rem:LeiNualart}
and the definition of $f$, we can see that: 
\begin{enumerate}
        \item[(i)] $Y_t=\ka_Hu_t(0)$; and 
        \item[(ii)] $u$ solves the stochastic PDE \eqref{SDEmodified}.
\end{enumerate}
We also remark that $g(x)=c_\al^22^{1/(2\al)}f(x)$.

We are assured by \eqref{eq:u:in:C}
that $Y\in\cap_{\gamma\in(0,H)}C^\gamma([0\,,t])$, up to a modification
[in the usual sense of stochastic processes]. Recall from \eqref{v}
the solution $v$ to the linear SPDE \eqref{SHE-linear}. 

Let $X$ be the $\fbm(H)$ from Proposition \ref{pr:bifBM:linear}
and choose and fix $t\in(0\,,T]$. Then
\begin{align}\notag
        \Theta &:= Y_{t+\ep} - Y_t - g(Y_t)(X_{t+\ep} - X_t)\\\notag
        &= Y_{t+\ep} - Y_t - c_\al^22^{1/(2\al)}f(Y_t)(X_{t+\ep} - X_t)\\\notag
        &= c_\al\left({u_{t+\ep}(0) - u_t(0) - f(c_\al u_t(0))
                \left[c_\al2^{1/(2\al)}X_{t+\ep} - c_\al2^{1/(2\al)}X_t\right]}\right)\\\notag
        &= c_\al \left( u_{t+\ep}(0) - u_t(0) - f(c_\al u_t(0)) (v_{t+\ep}(0) - v_t(0)) \right)\\
        &\hskip2in + c_\al f(c_\al u_t(0))(R_{t+\ep}-R_t).
\end{align}

We proved, earlier in Remark \ref{R:smooth_moments}, 
that $\|R_{t+\ep}-R_t\|_k\le A_{k,t}\,\ep$. 
Because $f$ is Lipschitz continuous, H\"older's inequality and \eqref{eq:moments}
together imply that $\|c_\al f(c_\al u_t(0))(R_{t+\ep}-R_t)\|_k \le A_{k,t}\,\ep,$
whence we obtain the bound,
\begin{equation}
        \sup_{t\in(0,T]}\E(\Theta^2)\le A_T\ep^{2\mathcal{G}_H},
\end{equation}
from Theorem \ref{th:main}. Since $\mathcal{G}_H>H$---see
Remark \ref{rem:GH>H}---the preceding displayed bound and Chebyshev's inequality
together imply that for every $\ep\in(0\,,1)$,  $\delta > 0$, 
and $b\in(H,\mathcal{G}_H)$,
\begin{align}
        \P\left\{{\left|{\frac{Y_{t+\ep} - Y_t}{X_{t+\ep} - X_t}
                - g(Y_t)}\right| > \de}\right\}
                &= \P\left\{{{\frac{|\Theta|}{|X_{t+\ep} - X_t|} }} > \de\right\}\\\notag
        &\le A_T\,\ep^{2(\mathcal{G}_H-b)}
                + \P\left\{|X_{t+\ep} - X_t|<\frac{\ep^b}{\delta} \right\}. 
\end{align}
The first term converges to zero as $\ep\to 0^+$ since $b<\mathcal{G}_H$. It remains
to prove that the second term also vanishes as $\ep\to 0^+$.
But since $X$ is $\fbm(H)$, the increment $X_{t+\ep}-X_t$
has the same distribution as $\ep^HZ$ where $Z$ is a standard 
normal random variable. Therefore,
\begin{equation}
        \sup_{t\in(0,T]}\P\left\{|X_{t+\ep} - X_t|<\frac{\ep^b}{\delta} \right\} =
        \P\left\{ |Z| \le \frac{\ep^{b-H}}{\delta}\right\},
\end{equation}
which goes to zero as $\ep\to 0^+$ since $b>H$.
\qed

\begin{spacing}{0.2}
\begin{small}
\end{small}\end{spacing}
\vskip.4in

\begin{small}
\noindent\textbf{Davar Khoshnevisan} (\texttt{davar@math.utah.edu})

\noindent Dept.\ Mathematics, Univ.\ Utah,
        Salt Lake City, UT 84112-0090

\medskip

\noindent\textbf{Jason Swanson} (\texttt{jason@swansonsite.com})

\noindent Dept.\ Mathematics, Univ.\ Central Florida,
        Orlando, FL 32816-1364

\medskip

\noindent\textbf{Yimin Xiao} (\texttt{xiao@stt.msu.edu})

\noindent Dept.\  Statistics \&\ Probability, 
        Michigan State Univ., East Lansing, MI 48824-3416

\medskip

\noindent\textbf{Lianng Zhang} (\texttt{lzhang81@cs.utah.edu})

\noindent Dept.\  Statistics \&\ Probability, 
        Michigan State Univ., East Lansing, MI 48824-3416

\end{small}

\end{document}